\documentclass[12pt]{amsart}
\usepackage[all]{xy}
\usepackage{amssymb}
\newcommand{\Mono}{\ar@{>->}[r]}
\newcommand{\Epi}{\ar@{->>}[r]}
\newcommand{\Z}{{\mathbb Z}}
\newcommand{\C}{{\mathbb C}}
\newcommand{\N}{{\mathbb N}}
\newcommand{\Q}{{\mathbb Q}}
\newcommand{\F}{{\mathbb F}}
\newcommand{\Fp}{{\mathbb{F}_p}}
\newcommand{\pcom}{_{p}^{\wedge}}
\newcommand{\padic}{\Z_p}
\newcommand{\iso}{\cong}
\newcommand{\con}{\equiv}

\newcommand{\hocolim}{\operatornamewithlimits{hocolim}}
\newcommand{\Hom}{\operatorname{Hom}\nolimits}
\newcommand{\Coker}{\operatorname{Coker}\nolimits}
\newtheorem{Thm}{Theorem}[section]
\newtheorem{Prop}[Thm]{Proposition}

\newtheorem{Lem}[Thm]{Lemma}

\theoremstyle{remark}
\newtheorem{Rm}[Thm]{Remark}
\newenvironment{psmatrix}{\left(\begin{smallmatrix}}{\end{smallmatrix}\right)}
\def\Matrix#1#2#3#4{\left(\begin{matrix}#1&#2\\#3&#4\end{matrix}\right)}
\def\psMatrix#1#2#3#4{\begin{psmatrix}#1&#2\\#3&#4\end{psmatrix}}
\newcommand{\kx}{k^{\times}}
\newcommand{\G}{G(k)}
\title[Cohomology of Kac-Moody groups]{Cohomology of Kac-Moody groups over a finite field}
\author{Jaume Aguad\'{e}}\author{Albert Ruiz}
\address{Departament de Matem\`atiques\\ Universitat Aut\`onoma de Barcelona\\ 08193 Cerdanyola del Vall\`es\\ Spain.}
\email{aguade@mat.uab.cat}\email{Albert.Ruiz@uab.cat}
\thanks{The authors are partially supported by grants FEDER-MICINN MTM2010-20692 and 2009SGR-1092.}
\begin{document}
\begin{abstract}We compute the mod $p$ cohomology algebra of a family of infinite discrete Kac-Moody groups of rank two defined over
finite fields of characteristic different from $p$.\end{abstract} \maketitle

\section{Introduction}\label{Introduction}

A functorial definition of discrete Kac-Moody groups over any commutative ring was established by Tits in 1987 (\cite{Tits}) and
there is now a great deal of interest in these objects, as illustrated by the recent monograph by R\'{e}my (\cite{Remy}), 
and by some of the references it
contains. The work of Kac-Peterson (\cite{KacPete}) and Kitchloo (\cite{tesiNitu}) on the classifying spaces
$BK(\C)$ of the topological Kac-Moody groups over the complex field were the starting point for the study of these objects from
the point of view of the so-called \emph{homotopical group theory} (cf.\ the address by Grodal at the 2010 ICM, \cite{Grodal}).
Since then, some significant progress has been made in this area. For instance, we know that the mod $p$ cohomology of $BK(\C)$
is noetherian (\cite{CarlesNitu}). In the case of rank two, the mod $p$ cohomology of $BK(\C)$ has been explicitely computed
(\cite{Kyoto}), and the self maps of $BK(\C)$ have been described and classified (\cite{Advances}).

The main results in this paper (Proposition \ref{trivial} and Theorem \ref{Main}) provide an explicit description of
the cohomology algebra $H^*(BG_{\mathcal{D}}(k);\Fp)$ for $G_{\mathcal{D}}(k)$ a discrete infinite Kac-Moody group of
rank two over a finite field $k$ of characteristic different from $p$ (under some restrictions, as stated in
\ref{Main}). To arrive to this result we need detailed descriptions of the root system (section \ref{roots}), the
parabolic subgroups and the L\'{e}vi subgroups (section \ref{parabolicsec}) of $G_{\mathcal{D}}(k)$. The modular
invariant theory of the dihedral group plays also a main role in our computation (sections \ref{Weylsec} and
\ref{Coker}). Some of these partial steps in the main result may also have interest in their own.

In the classic case of a finite Weyl group, Friedlander (\cite{Friedlander}) discovered that there is a homotopy equivalence
between
the $p$-completion of the classifying space of the Chevalley group over the field $\F_q$ and the homotopy fixed points of an
unstable Adams map $\psi^q$ defined on the classifying space of the corresponding compact connected Lie group. In a final section
of this paper we see that this result does not generalize to the Kac-Moody group case (infinite Weyl group). This suggests a
discrepancy between the algebraic and the homotopical definitions of Kac-Moody groups over finite fields, in the case of an
infinite Weyl group.

We want to acknowledge some helpful correspondence and conversations with Alejandro Adem and Bertrand R\'{e}my on some
details of our work. The first author is grateful to the Pacific Institute for the Mathematical Sciences in Vancouver
for its hospitality during the preparation of this paper.

\begin{Rm}
After finishing this paper, we became aware of results in J.~Foley's PhD Thesis (\cite{Foley}), having a 
significant overlap with the results of this work and, in some points, reaching beyond 
(cf. Remark \ref{acyclicgeneral}).
\end{Rm}

\section{Kac-Moody groups over a field}\label{Tits}

In this preliminary section we recall the basic notions and notations of Kac-Moody groups over a field. Our main references are
the original paper by Tits (\cite{Tits}) and the book by R\'{e}my (\cite{Remy}).

A Kac-Moody group functor depends on a set of data $\mathcal{D}$ consisting of a $n\times n$ generalized Cartan matrix $A$, a
free abelian group $\Lambda$, and elements $\alpha_1,\ldots,\alpha_n\in\Lambda$ and
$h_1,\ldots,h_n\in\Lambda^\vee=\Hom(\Lambda,\Z)$ such that $\langle h_i,\alpha_j\rangle=A_{ij}$. From these data
$\mathcal{D}$ one can construct a functor $G_{\mathcal{D}}(-)$ from commutative rings to groups which coincides with the
classic Chevalley-Demazure functor if $A$ is a Cartan matrix. When restricted to fields, the functor
$G_{\mathcal{D}}(-)$ can be characterized by a small set of axioms.

Let $k$ be a field and let $k^+$ and $\kx$ be the additive and multiplicative groups of $k$, respectively. The abelian group
$T=\Hom(\Lambda,\kx)$ plays the role of a maximal torus in $G_{\mathcal{D}}(k)$ through a monomorphism $\eta\colon T\to
G_{\mathcal{D}}(k)$. For each $1\le i\le n$ there is a homomorphism $\phi_i\colon SL_2(k)\to G_{\mathcal{D}}(k)$ as well as two
monomorphisms $x_i^+,x_i^-\colon k^+\to G_{\mathcal{D}}(k)$ such that $x_i^+(r)=\phi_i\psMatrix1r01$ and
$x_i^-(r)=\phi_i\psMatrix10{-r}1$. The kernel of each $\phi_i$ is central in $SL_2(k)$. The group $G_{\mathcal{D}}(k)$ is
generated by $T$ and the images of all $\phi_i$. Each $\phi_i$ sends diagonal matrices into the maximal torus $T$ as follows: for
any $r\in\kx$ we have
$$
\hspace*{5cm}\phi_i\psMatrix r00{r^{-1}}(\lambda)=r^{h_i(\lambda)}.\hspace*{5.5cm}(\ast)
$$
$T$ normalizes $U_i^+=x_i^+(k^+)$ and $U_i^-=x_i^-(k^+)$ by the formula
$$
\hspace*{4.5cm}t\,x_i^\pm(r)\,t^{-1}=x_i^\pm(t(\alpha_i)^{\pm1}r)\hspace*{5cm}(\ast\ast)
$$
for any $r\in k$.

The equations $\omega_i(a_j)=a_j-A_{ij}a_i$, $i,j=1,\ldots,n$ define an action of a group $W$ (the Weyl group) on $\Q
a_1\oplus\cdots\oplus\Q a_n$. The orbit $\Phi=W\{a_1,\ldots,a_n\}$ is the root system of $G_{\mathcal{D}}(k)$. Every
root is an integral linear combination of $\{a_1,\ldots,a_n\}$ with coefficients all positive or all negative. We talk
of positive roots $\Phi^+$ and negative roots $\Phi^-$. $W$ is a finite group if and only if the matrix $A$ is a Cartan
matrix (i.\ e.\ $A$ is the product of a diagonal matrix and a positive definite symmetric matrix). The root system
depends only on the matrix $A$ and not on the full data $\mathcal{D}$. The Weyl group $W$ acts also on $\Lambda$ (and on
$T$) by by $\omega_i(\lambda)=\lambda-\langle h_i,\lambda\rangle\,\alpha_i$, for $\lambda\in\Lambda$ and
$i=1,\ldots,n$.

$G_{\mathcal{D}}(k)$ is a group with a double $BN$-pair. There are Borel subgroups $B^+$, $B^-$, standard parabolic subgroups
$P_I^+$, $P_I^-$ for any $I\subset\{1,\ldots,n\}$, root groups $U_a$ for any $a\in\Phi$, and all the rich theory of double
BN-pairs applies.

From a topological point of view, one of the most remarkable properties of Kac-Moody groups with infinite Weyl group,
which does not hold in the classic case of a finite Weyl group, is the following.

\begin{Thm}\label{Mitchell}
(\cite{Mitchell}) Let $G$ be a Kac-Moody group with an infinite Weyl group and let $\mathcal{C}$ denote the poset of proper
standard parabolic subgroups $P_I$ of $G$. Then there is a homotopy equivalence $BG\simeq\hocolim_{\mathcal{C}}BP_I$.\qed
\end{Thm}

The purpose of this paper is to compute $H^*(BG_{\mathcal{D}}(k);\Fp)$ when the free
abelian group $\Lambda$ is of rank two, $k$ is a finite field, the Weyl group of $G_{\mathcal{D}}(k)$ is infinite and
$p$ is an odd prime different from the characteristic of the field $k$.

Throughout this paper we fix the following notations:

\begin{enumerate}
\item $p,\ell$ are different primes, $p$ is odd and $k$ is a finite field of order $q$ and characteristic $\ell$.
\item $a,b$ are positive integers such that $ab\ge4$ and $A$ is the generalized Cartan matrix $$A=\Matrix2{-a}{-b}2.$$
Without loss of generality, we assume $a\le b$.
\item $n_i,m_i,s_i,t_i$ for $i=1,2$ are positive integers such that
$$\Matrix{s_1}{t_1}{s_2}{t_2}\Matrix{n_1}{n_2}{m_1}{m_2}=A.$$
We denote $\Delta=s_1t_2-s_2t_1$, $\nabla=n_1m_2-n_2m_1$, so that $\Delta\nabla=4-ab$.
\item $\Lambda$ is a free abelian group of rank two $\Lambda=\Z e_1\oplus\Z e_2$ and $\Lambda^\vee=\Z
e_1^\vee\oplus\Z e_2^\vee$ is its $\Z$-dual.
\item $\alpha_1,\alpha_2\in\Lambda$ and $h_1,h_2\in\Lambda^\vee$ are the elements $\alpha_i=n_ie_1+m_ie_2$,
$h_i=s_ie_1^\vee+t_ie_2^\vee$, $i=1,2$.
\item $G$ is the Kac-Moody group $G_{\mathcal{D}}(k)$ where $\mathcal{D}$ consists of the matrix $A$, the lattice
$\Lambda$ and the elements $\{\alpha_1,\alpha_2\}$ and $\{h_1,h_2\}$.
\end{enumerate}

In particular, we want to emphasize that through all this paper the prime $p$ is always assumed to be odd and different
from the characteristic of the field $k$.

\section{The root system in rank two}\label{roots}

In this section we will study the root system $\Phi$ associated to the generalized Cartan matrix $A$. Recall that $\Phi$ is the
orbit of the basis vectors of $\Q^2$ under the action of the infinite dihedral group $W=\langle\omega_1,\omega_2\rangle$ acting
as $$\omega_1=\Matrix{-1}a01,\quad\omega_2=\Matrix10b{-1}.$$ This root system $\Phi$ has been studied in \cite{Tits-Belgique},
\cite{KacPete}, \cite{tesiNitu}, \cite{tesiAlbert}. The lattice $\Lambda$ and the elements
$\alpha_1,\alpha_2\in\Lambda$, $h_1,h_2\in\Lambda^\vee$ do not play any role in this section.

If we write $\tau=\omega_1\omega_2$, then the matrices in $W$ can be described as follows (cf.\ \cite{tesiAlbert},
\cite{tesiNitu})
$$\tau^n=\Matrix{d_{2n+1}}{-d_{2n}}{c_{2n}}{-c_{2n-1}},\quad
\tau^n\omega_1=\Matrix{-d_{2n+1}}{d_{2n+2}}{-c_{2n}}{c_{2n+1}}$$ where the integers $c_n$, $d_n$ are defined inductively
for any
integer $n$ as follows:
\begin{align*}
c_0=d_0=0,\quad &c_1=d_1=1,\\ c_{n+1}=bd_n-c_{n-1},\quad &d_{n+1}=ac_n-d_{n-1}.\end{align*}

The following proposition lists many properties of these integers that we will use in this section.

\begin{Prop}\label{longlist}
\noindent
\begin{enumerate}
\item $c_{2n+1}=d_{2n+1}$, $bd_{2n}=ac_{2n}$.
\item $c_{2n}\equiv0\,( b)$, $c_{2n+1}\equiv(-1)^n\,( b)$, $d_{2n}\equiv0\,(a)$,
$d_{2n+1}\equiv(-1)^n\,(a)$.
\item $c_{-n}=-c_n$, $d_{-n}=-d_n$.
\item If $a=1$ and $b=4$, then $d_{2n}=n$, $c_{2n}=4n$, $c_{2n+1}=d_{2n+1}=2n+1$.
\item If $ab>4$, let $\zeta>1$ be a real root of $X^4-(ab-2)X^2+1$. Then
$$c_n=\begin{cases}\displaystyle\frac{\zeta^{2n}-1}{\zeta^{n-1}(\zeta^2-1)}\quad\text{ $n$ odd}\\ \\
\displaystyle\frac{b(\zeta^{2n}-1)}{\zeta^{n-2}(\zeta^4-1)}\quad\text{ $n$ even}\end{cases}$$
\item $c_n,d_n>0$ for $n>0$.
\item For $ab>4$ the function $f(n)=c_{2n+1}-c_{2n-1}$, $n>0$, is strictly increasing.
\item If $a,b>1$ then the sequences $\{c_n\}$, $\{d_n\}$ are strictly increasing.
\item The sequences $\{c_{2n}\}$, $\{d_{2n}\}$, $\{c_{2n+1}\}$, $\{d_{2n+1}\}$ are strictly increasing.
\item If $(a,b)\ne(1,4)$ then $d_{2n}<d_{2n+1}<d_{2n+4}$ for $n\ge0$.
\item If $a=1$ then for $n>1$ we have $d_{2n+1}<\min\{(b-1)d_{2n},(b-2)d_{2n-1}\}$.
\end{enumerate}
\end{Prop}

\begin{proof}
Most of these properties are either evident or can be easily proven by induction. To prove (7), use (5) to obtain a
formula for $f(n)$ and then check that $(d/dx)f(x)>0$ (cf.\ \cite{tesiAlbert}). To prove (8) by induction notice that if
$n$ is odd then
$$c_{n+1}=bd_n-c_{n-1}=bc_n-c_{n-1}\ge2c_n-c_{n-1}>c_n$$ and if $n$ is even then
$$c_{n+1}=bd_n-c_{n-1}=ac_n-c_{n-1}\ge2c_n-c_{n-1}>c_n$$ and similarly with $\{d_n\}$.

To prove (9) it is enough to consider the case $a=1$. If $b=4$, the result follows from (4). If $b\ge5$ we know from (7)
that $\{c_{2n+1}\}$ is strictly increasing. Then, $d_{2n}=(c_{2n+1}+c_{2n-1})/b$ and $c_{2n}=bd_{2n}$ are also strictly
increasing.

In (10) we can also assume $a=1$ and $b\ge5$. If we write $d_{2n+1}-d_{2n}$ a a function of $\zeta$ using (5) then we
see that $d_{2n}<d_{2n+1}$. We have $$\zeta^4=(b-2)\,\zeta^2-1\ge3\,\zeta^2-1>\zeta^2+1.$$ Then if $k=2n+1$ we see that
$\zeta^{2k+2}>\zeta^{2k}+\zeta^{2k-2}-1$ and so $$\zeta^{2k+6}-1>\zeta^2(\zeta^{2k}-1)(\zeta^2+1)$$ which implies
$d_{k+3}>d_k$.

If $a=1$ we have $$d_{2n+1}=bd_{2n}-d_{2n-1}=(b-1)d_{2n}-d_{2n-1}+d_{2n}$$ and $$d_{2n}=d_{2n-1}-d_{2n-2}<d_{2n-1}$$ for
$n>1$. This proves the first inequality in (11). To get the second one, observe
\begin{equation*}\begin{split}d_{2n+1}&=bd_{2n}-d_{2n-1}=(b-1)d_{2n-1}-bd_{2n-2}\\
&=(b-2)d_{2n-1}+bd_{2n-2}-d_{2n-3}-bd_{2n-2}<(b-2)d_{2n-1}\end{split}\end{equation*} if $n>1$.
\end{proof}

In this section, it is convenient to denote $\{u_0,v_0\}$ the standard basis of $\Q^2$. We will use the following
notation ($i\ge0$):
\begin{align*}
u_i=\tau^iu_0=(d_{2i+1},c_{2i}),\quad
&v_i=\tau^{-i}v_0=(d_{2i},c_{2i+1}),\\
\overline{u}_i=\tau^i\omega_1v_0=(d_{2i+2},c_{2i+1}),\quad
&\overline{v}_i=\tau^{-i}\omega_2u_0=(d_{2i+1},c_{2i+2}).
\end{align*}

It is clear that these are the positive roots $\Phi^+$ of the infinite root system $\Phi$. We write
$$
\mathcal{A}=\big\{u_i,\overline{u}_i\,|\,i\ge0\big\},\quad
\mathcal{B}=\big\{v_i,\overline{v}_i\,|\,i\ge0\big\}.
$$

Recall (\cite{Tits}) that a set of roots $\Psi$ is called \emph{pre-nilpotent} if there are elements $\omega,\omega'\in
W$ such that $\omega\Psi\subset\Phi^+$ and $\omega'\Psi\subset\Phi^-$.

\begin{Prop}\label{prenilpotent}
If $e,w\in\mathcal{A}$ (resp.\ $\mathcal{B}$), then the pair $\{e,w\}$ is pre-nilpotent.
\end{Prop}

\begin{proof} If $e,w\in\mathcal{A}$ and $N>0$ is large enough then we have $\tau^{-N}e,\tau^{-N}w\in\Phi^-$. If
$e,w\in\mathcal{B}$ then we may use $\tau^N$ for $N>0$ large enough.
\end{proof}

\begin{Prop}\label{noroots}
If $e,w\in\mathcal{A}$ (resp.\ $\mathcal{B}$) and $a>1$, then $e+w\notin\Phi$.
\end{Prop}

\begin{proof}
If $e,w\in\mathcal{A}$ and $e+w\in\mathcal{B}$ then for $N>$ large enough $\tau^N(e+w)\in\Phi^-$ while
$\tau^N(e),\tau^N(w)\in\Phi^+$. Hence, if $e+w\in\Phi$ then $e+w\in\mathcal{A}$. Moreover, $W$ leaves
invariant the quadratic form $$Q(X,Y)=bX^2+aY^2-abXY$$ and $Q(e)=Q(w)=Q(e+w)$ is impossible. Hence, any relation
$e+w\in\Phi$ should be one of these ($i,j,k\ge0$):
1) $u_i+u_j=\overline{u}_k$,
2) $u_i+\overline{u}_j=u_k$,
3) $u_i+\overline{u}_j=\overline{u}_k$,
4) $\overline{u}_i+\overline{u}_j=u_k$.

Each of these equalities can be translated into an equality between coefficients $c_n$, $d_n$ as follows:
\begin{align*}
1)\quad&(d_{2i+1},c_{2i})+(d_{2j+1},c_{2j})=(d_{2k+2},c_{2k+1})\\
2)\quad&(d_{2i+1},c_{2i})+(d_{2j+2},c_{2j+1})=(d_{2k+1},c_{2k})\\
3)\quad&(d_{2i+1},c_{2i})+(d_{2j+2},c_{2j+1})=(d_{2k+2},c_{2k+1})\\
4)\quad&(d_{2i+2},c_{2i+1})+(d_{2j+2},c_{2j+1})=(d_{2k+1},c_{2k})
\end{align*}

In each case, reducing modulo $a$ or $b$ and applying \ref{longlist}(2) we get a contradiction. If we start with
$e,w\in\mathcal{B}$ and assume $e+w\in\Phi$, then applying $-\tau^N$ for $N$ large enough we obtain a relation $e'+w'\in
\Phi$ with $e',w'\in\mathcal{A}$ which we have seen is not possible.
\end{proof}

The above result fails for $a=1$ since in this case one sees immediately that
$$\overline{u}_0+\overline{u}_1=u_1,\quad\overline{v}_0+\overline{v}_1=v_1$$ and so, applying $\tau^{\pm i}$ we
obtain relations $$\overline{u}_i+\overline{u}_{i+1}=u_{i+1},\quad\overline{v}_i+\overline{v}_{i+1}=v_{i+1}.\quad(*)$$

\begin{Prop}\label{rootsa=1}
If $a=1$, $b>4$, $e,w\in\mathcal{A}$ (resp.\ $\mathcal{B}$) and $e+w=f\in\Phi$, then the equality $e+w=f$ is as in $(*)$
above.
\end{Prop}

\begin{proof}
Assume $e+w=f\in\Phi$ with $e,w\in\mathcal{A}$, the case of $e,w\in B$ being equivalent. The same arguments as in
\ref{noroots} imply that $e+w=f$ has to be one of these equalities ($i,j,k\ge0$):
3) $u_i+\overline{u}_j=\overline{u}_k$,
4) $\overline{u}_i+\overline{u}_j=u_k$.

Notice that if $ijk>0$ then $\tau^{-1}$ gives a relation of the same type with smaller subscripts. Hence, we can assume
that at least one of the subscripts $i,j,k$ is equal to zero.

Equation 3) implies
$$
d_{2i+1}+d_{2j+2}=d_{2k+2},\quad
c_{2i}+c_{2j+1}=c_{2k+1}. $$
Hence $d_{2k+2}>d_{2j+2}$ and since the sequence $\{d_{2n}\}$ is strictly increasing, we
get $k>j$. Also, since the sequence $\{c_{2n+1}\}$ is strictly increasing, we get $i>0$. Hence, we can assume $j=0$
which amounts to
$$
d_{2i+1}+1=d_{2k+2},\quad c_{2i}+1=c_{2k+1}.
$$
Then,
$$
bd_{2i}-d_{2i-1}=d_{2i+1}=d_{2k+2}-1=d_{2k+1}-d_{2k}-1=bd_{2i}-d_{2k}.
$$
Using \ref{longlist}(10) we have $d_{2i+2}>d_{2i-1}=d_{2k}<d_{2k+1}$ and so $i=k$ which is impossible since this would
imply $1+d_{2k+1}=d_{2k+2}=d_{2k+1}-d_{2k}$.

On the other side, equation 4) implies
$$
d_{2i+2}+d_{2j+2}=d_{2k+1},\quad
c_{2i+1}+c_{2j+1}=c_{2k1}. $$
In particular, $k>0$ and without loss of generality we can assume $i=0$. We have
$$
1+d_{2j+2}=d_{2k+1},\quad 1+c_{2j+1}=c_{2k}
$$
and then
\begin{align*}
1+d_{2j+2}=&d_{2k+1}=c_{2k}-d_{2k-1}=1+d_{2j+1}-d_{2k-1}\\
d_{2j+2}=&d_{2j+1}-d_{2j}
\end{align*}
and so $d_{2j}=d_{2k-1}<d_{2k+2}$ (using \ref{longlist}(10)) and $j<k+1$ because $\{d_{2n}\}$ is strictly increasing. On
the other hand, also by \ref{longlist}(10), we have $d_{2k-2}<d_{2k-1}=d_{2j}$ and since $\{d_{2n}\}$ is strictly
increasing we have $k-1<j$. Hence, $j=k$. Then, $c_{2k}=1+c_{2k+1}=1+c_{2k}-c_{2k-1}$ which yields $k=1$.
\end{proof}

The only remaining case is $a=1$, $b=4$. In this case, the roots are explicitly computed in \ref{longlist}(4) and the
following result can be easily obtained.

\begin{Prop}\label{rootsa1b4}
If $a=1$, $b=4$, $e,w\in\mathcal{A}$ (resp.\ $\mathcal{B}$) and $e+w=f\in\Phi$, then the equality $e+w=f$ is
$$\overline{u}_i+\overline{u}_{i+2k-1}=u_{i+k}$$ for some $i\ge0, k>0$ (resp.\
$\overline{v}_i+\overline{v}_{i+2k-1}=v_{i+k}$).\qed
\end{Prop}

To summarize, we have seen that in the root system $\Phi$ the sum of two positive roots \emph{of the same
type} ($\mathcal{A}$ or $\mathcal{B}$) is almost never a root. The only exceptions are
\begin{align*}
\overline{u}_i+\overline{u}_{i+1}&=u_{i+1},\; i\ge0\text{ which occur for $a=1$;}\\
\overline{u}_i+\overline{u}_{i+2k-1}&=u_{i+k},\; i\ge0,\,k>0,\text{ which occur for $(a,b)=(1,4)$}
\end{align*}
for type $\mathcal{A}$ and similarly for type $\mathcal{B}$. Our final step in this section is to compute, in each of
the cases $e+w\in\Phi$ above, the sets $$(\N e+\N w)\cap\Phi.$$

\begin{Prop}\label{rootsnu+mv}
If $e,w\in\mathcal{A}$ (resp.\ $\mathcal{B}$) and $e+w\in\Phi$, then $(\N e+\N w)\cap\Phi=\{e+w\}$.
\end{Prop}

\begin{proof} We know that we can assume $a=1$. Consider first the case $n\overline{u}_0+m\overline{u}_1=f\in\Phi$,
$n,m>0$. Then, there are only two possibilities: either $f=\overline{u}_i$ or $f=u_i$. In the first case, we can solve
$n\overline{u}_0+m\overline{u}_1=\overline{u}_i$ for $n,m$ and get
\begin{align*}
n&=(b-1)\,d_{2i+2}-(b-2)\,d_{2i+1}\\
m&=d_{2i+1}-d_{2i+2}.
\end{align*}
Then, $n=d_{2i+1}-(b-1)\,d_{2i}$ and then \ref{longlist}(11) yields $n=0$. If $n\overline{u}_0+m\overline{u}_1=u_i$
then we can also solve for $n,m$:
\begin{align*}
n&=(b-1)\,d_{2i+1}-(b-2)\,c_{2i}\\
m&=c_{2i}-d_{2i+1}.
\end{align*}
Hence, $n=d_{2i+1}-(b-2)\,d_{2i-1}$ and \ref{longlist}(11) yields $n=m=1$.

In the general case of $n\overline{u}_i+m\overline{u}_{i+1}\in\Phi$ we can apply $\tau^{-1}$ enough times to get
$n\overline{u}_0+m\overline{u}_1\in\Phi$ which we have already ruled out.

It remains only the case $(a,b)=(1,4)$ where we have to consider $$n\overline{u}_i+m\overline{u}_{i+2k-1}\in\Phi,\quad
i\ge0,\,k>0.$$ Using $\tau^{-i}$ we can assume that $i=0$. Recall that the explicit values of $c_n$ and $d_n$ are given
by \ref{longlist}(4). Then, the solutions for $n\overline{u}_0+m\overline{u}_{2k-1}=\overline{u}_j$ are
$$n=1-\frac{j}{2k-1},\quad m=\frac{j}{2k-1}$$ which are not possible. Also, the solutions for
$n\overline{u}_0+m\overline{u}_{2k-1}=u_j$ are $$n=2-\frac{2j-1}{2k-1},\quad m=\frac{2j-1}{2k-1}$$ which are only
possible for $n=m=1$. The proposition is proven.
\end{proof}

\section{The parabolic subgroups and the L\'{e}vi decomposition}\label{parabolicsec}

In the case of the Kac-Moody group $G$ that we are considering, the poset of the proper (positive) parabolic subgroups consists
of only three groups $P_{\varnothing}=B^+$, $P_1$, $P_2$ and \ref{Mitchell} reduces to the fact that $BG$ is the homotopy colimit
of the diagram $BP_1\leftarrow BP_{\varnothing}\rightarrow BP_2$. This is equivalent to saying that $G$ is the amalgamated
product of $P_1$ and $P_2$ over $B^+$ (\cite{Brown}, Th.\ 7.3). In this section we want to investigate the group theoretical
structure of these groups.

Like in the case of a finite Weyl group, there is a L\'{e}vi decomposition for parabolic subgroups of Kac-Moody groups
(\cite{Remy}, 6.2) $P_I=H_I\ltimes V_I$. In particular, we have that $P_{\varnothing}=B=T\ltimes U^+$ where $U^+$ is the subgroup
generated by all positive root groups. Using the notations of section \ref{roots}, we have:

\begin{Prop}\label{A*B} (\cite{KacPete}, proposition 4.3 and remark 2)
Let $U_{\mathcal{A}}$ (resp.\ $U_{\mathcal{B}}$) be the subgroup generated by all root groups $U_e$ for $e\in\mathcal{A}$ (resp.\
$e\in\mathcal{B}$). Denote $E=\oplus_0^\infty k^+$. Then,
\begin{enumerate}
\item $U^+=U_{\mathcal{A}}\ast U_{\mathcal{B}}$.
\item If $a>1$, then both $U_{\mathcal{A}}$ and $U_{\mathcal{B}}$ are isomorphic to $E$.
\item If $a=1$, then both $U_{\mathcal{A}}$ and $U_{\mathcal{B}}$ are extensions of $E$ by $E$.
\end{enumerate}
\end{Prop}

\begin{proof}
(1) is in \cite{KacPete} 4.3. (2) and (3) follow from the analysis in section \ref{roots}. If $a>1$ we have seen that the sum of
two roots in $\mathcal{A}$ (resp.\ $\mathcal{B}$) is not a root. Hence, $U_{\mathcal{A}}$ and $U_{\mathcal{B}}$ are abelian
groups generated by countably many root groups, each one isomorphic to $k^+$. If $a=1$ we have seen that in some cases, the sum
of two roots in $\mathcal{A}$ can be a root (cf.\ \ref{rootsa=1}, \ref{rootsa1b4}). However, the subgroup of $U_{\mathcal{A}}$
generated by the root groups $U_{u_i}$ for $i\ge0$ is abelian, isomorphic to $E$ and normal in $U_{\mathcal{A}}$. Also,
it follows from \ref{rootsa=1}, \ref{rootsa1b4} and \ref{rootsnu+mv} that the quotient of $U_{\mathcal{A}}$ by this
normal subgroup is also abelian and isomorphic to $E$. The same holds for $U_{\mathcal{B}}$.
\end{proof}

To determine the structure of the subgroups $V_I$ we need the following lemma.

\begin{Lem}\label{coverings}
Let $K$ and $M$ be groups and assume $N$ is a normal subgroup of index $r$ in $K$. Consider the inclusion $N\ast M < K \ast M$.
Then the normal closure of $N\ast M$ in $K\ast M$ is isomorphic to the free product $ N \ast (\ast_{i=1}^r M)$.
\end{Lem}

\begin{proof}
There is an easy topological proof for this lemma. The inclusion $N\vartriangleleft K$ can be realized topologically by
a pointed map of classifying spaces $\pi\colon BN\to BK$ which is a regular $r$-fold covering. Let $\{x_1,\ldots,x_r\}$ be
the fiber of $\pi$ over the base point of $BK$ and consider the map $\pi'\colon \widetilde{X}\to X$ defined as follows.
$X=BK\vee BM$, $\widetilde{X}=BN\vee(BM)_1\vee\cdots\vee(BM)_r$ where each $(BM)_i$ is glued to $BN$ at the point
$x_i$ and the map $\pi'$ is defined so that $\pi'|_{BN}=\pi$ and $\pi'|_{(BM)_i}=\text{id}$.

Then, it is clear that $\pi'$ is also a regular covering and $\pi'^*$ identifies $\pi_1(\widetilde{X})\iso N \ast
(\ast_{i=1}^r M)$ to a normal subgroup of $K\ast M$. Hence, the normal closure of $N\ast M$ in $K\ast M$ is contained
in $\pi_1(\widetilde{X})$.

On the other side, if $\gamma_i$ is a path in $BN$ from the base point to the
point $x_i$, then the copy of $M$ in $\ast_{i=1}^r M$ corresponding to the summand $(BM)_i$ is sent by $\pi'^*$ to
the conjugate $[\pi\gamma_i]^{-1}M[\pi\gamma_i]$ in $K\ast M$. Hence, $\pi_1(\widetilde{X})$ is contained in the normal
closure of $N\ast M$ in $K\ast M$.
\end{proof}

We are now ready to prove that for our cohomology computations we can replace the parabolic subgroups $P_I$ by the L\'{e}vi
subgroups $H_I$.

\begin{Prop}\label{acyclic}
For each $I\varsubsetneq\{1,2\}$, the homomorphism $P_I\rightarrow H_I$ induces an isomorphism in cohomology with coefficients in
$\Fp$.
\end{Prop}

\begin{proof}
$P_I$ is an extension of $H_I$ by $V_I$, hence it is enough to prove that the groups $V_I$ are $p$-acyclic.
Recall that we are assuming that $p$ is prime to the order of $k$ and so the group $k^+$ is $p$-acyclic.
For $I=\varnothing$ we have $V_{\varnothing}=U^+$ and Proposition \ref{A*B} shows that $U^+$ is indeed $p$-acyclic.
Consider $V_1$. According to \cite{Remy} 6.2, $V_1$ is the normal closure in $U^+$ of the subgroup generated by the root
groups $U_e$ for $e\in\Phi^+-\{u_0\}$. \ref{A*B} and \ref{coverings} allows us to compute this normal closure and we see
that it is also mod $p$ acyclic.
\end{proof}

\begin{Rm}\label{acyclicgeneral}
The mod $p$ triviality of the unipotent subgroups of Kac-Moody groups over finite fields of characteristic 
different from $p$ has been proved independently in \cite{Foley}.
\end{Rm}

This last result implies that if we take coefficients in $\Fp$, then
the cohomology of $G$ is isomorphic to the cohomology of the group $H_1\ast_TH_2$. Notice that the hypothesis
$p\ne\ell$ is crucial.

The remainder of this section is devoted to study the structure of the groups $H_i=\langle T,U_i^+,U_i^-\rangle=\langle
T,\phi_i(SL_2(k))\rangle$ as well as the homomorphism $\eta\colon T\to H_i$. Recall from section \ref{Tits} the meaning of
the integers $n_i$, $m_i$, $s_i$, $t_i$ for $i=1,2$. Recall also the action of the Weyl group $W$ on $\Lambda$ given
by $\omega_i(\lambda)=\lambda-\langle h_i,\lambda\rangle\,\alpha_i$, for $\lambda\in\Lambda$.
Since the analysis for $H_1$ is the same as for $H_2$, we omit all subscripts $i=1,2$ and we write $H$, $n$, $m$, $s$,
$t$, $\omega$ to simplify the typography.

The homomorphism $\phi\colon SL_2(k)\to H$ can have a non trivial central kernel. We say that $H$ is \emph{monic} if $\phi$ is
injective. The property $(\ast)$ in section \ref{Tits} shows that $H$ is monic if and only if $\ell=2$ or $s,t$ are relatively
prime.

The proof of the following proposition is straightforward.

\begin{Prop}\label{extensionsmonic}
Assume $H$ is monic and consider $\psi\colon \kx\to\kx\times\kx$ given by $\psi(\zeta)=(\zeta^s,\zeta^t)$. Then, there is a
split exact sequence of abelian groups
$$\xymatrix{\kx\;\Mono^-\psi&\kx\times\kx\Epi^-\pi&\kx}$$ with $\pi(\zeta,\tau)=\zeta^{-t}\tau^s$. If $s,t$ are
relatively prime, a section is given by $\sigma(\zeta)=(\zeta^{-\mu},\zeta^{\lambda})$ where $\lambda,\mu$ are integers
such that $\lambda s+\mu t=1$. If $s,t$ are both even, then $\ell=2$ and a section is given by
$\sigma(\zeta)=((\zeta^{1/2})^{-m},(\zeta^{1/2})^{n})$.
\qed
\end{Prop}

\begin{Prop}\label{semidirectmonic}
 Let $E=SL_2(k)\rtimes\kx$ with action given by $\zeta\cdot\psMatrix
xyzt=\psMatrix x{\zeta^r y}{\zeta^{-r}z}t$ for some integer $r$. Then, $E\iso GL_2(k)$ if $r$ is odd and $E\iso
SL_2(k)\times\kx$ if $r$ is even.
\end{Prop}

\begin{proof}
 If $r=2r'-1$, consider the homomorphism $$(M,\zeta)\mapsto M\psMatrix{\zeta^{r'}}00{\zeta^{1-r'}}$$ from
$SL_2(k)\rtimes\kx$ to $GL_2(k)$. If $r=2r'$, consider the homomorphism $$(M,\zeta)\mapsto
\left[M\psMatrix{\zeta^{r'}}00{\zeta^{-r'}},\zeta\right]$$ from $SL_2(k)\rtimes\kx$ to $SL_2(k)\times\kx$.
\end{proof}

Let $\overline{H}$ denote any of the groups $GL_2(k)$, $SL_2(k)\times\kx$ and let $\overline{\eta}$ be the homomorphism
$\overline{\eta}\colon T\to\overline{H}$ given by $\overline{\eta}(\zeta,\tau)=\psMatrix\zeta00\tau\in GL_2(k)$ and
$\overline{\eta}(\zeta,\tau)=\left[\psMatrix\zeta00{\zeta^{-1}},\tau\right]\in SL_2(k)\times\kx$.
Let $\overline\omega$ be the automorphism of $\Lambda$ given by $\overline\omega(e,v)=(v,e)$ if $\overline{H}=GL_2(k)$
and $\overline\omega(e,v)=(-e,v)$ if $\overline{H}=SL_2(k)\times\kx$.

We say that $H$ is
\emph{split} if $n,m$ are both even. With these notations, we can state the following structure theorem.

\begin{Prop}\label{parabolicmonic}
Assume $H$ is monic. Let $\overline{H}=SL_2(k)\times\kx$ if $H$ is split and $\overline{H}=GL_2(k)$ if it is not. Then,
\begin{enumerate}
\item There is an isomorphism $\gamma\colon H\iso\overline{H}$ such that the following diagram is commutative
$$\xymatrix{
T\ar[r]^{\eta}\ar[d]_{\gamma|_T}&H\ar[d]_\iso^\gamma\\
T\ar[r]^{\overline\eta}&\overline{H}
}
$$
 \item $\gamma|_T$ is induced by $M\colon \Lambda\to\Lambda$ given by $$M=\begin{cases}\Matrix{n/2}{-t}{m/2}{s}\text{ if $H$
is split}\\ \frac12\Matrix{n-t}{-n-t}{m+s}{-m+s}\text{ if $H$ is not split}.\end{cases}$$
\item $M^{-1}\omega M=\overline\omega$.
\end{enumerate}
\end{Prop}

\begin{proof}
Recall that $H=\langle\phi(SL_2(k)),T\rangle$ and $T$ normalizes $SL_2(k)$ according to the formula $(\ast\ast)$ in
section \ref{Tits}. Since $H$ is monic we have that $s,t$ are relatively prime and we can apply
\ref{extensionsmonic}.
We have a commutative diagram were each row is an exact sequence and the bottom row is the split exact sequence in
\ref{extensionsmonic}.
$$
\xymatrix{
SL_2(k)\;\Mono&SL_2(k)\rtimes\kx\Epi&\kx\\
SL_2(k)\;\ar@{=}[u]\Mono^-\phi&H\ar[u]^\delta\Epi^-\pi&\kx\ar@{=}[u]\\
\kx\ar[u]\;\Mono^-\phi&T\ar[u]^\eta\Epi^-\pi&\kx\ar@{=}[u]\\
\kx\;\ar@{=}[u]\Mono^-\psi&\kx\times\kx\ar[u]^\iso_\alpha\Epi^-\pi&\kx\ar@{=}[u]\ar@<1ex>[l]^-\sigma
}
$$
This proves that $H\iso SL_2(k)\rtimes\kx$. To compute the action, choose $\lambda,\mu$ such that $\lambda s+\mu t=1$
and let $r=\lambda m-\mu n$. It is easy to see that $r$ is even if and only if $n,m$ are both even, i.\ e.\ if $H$ is
split. Then we conclude that the action of $\kx$ on $SL_2(k)$ is given by $$\zeta\cdot\Matrix xyzt=\Matrix
x{\zeta^ry}{\zeta^{-r}z}t$$ and then \ref{semidirectmonic} proves the first part of this proposition.

To compute the matrix $M$, notice that $\delta$ is given by
$\delta(g)=(\phi^{-1}(g\,\sigma\pi(g)^{-1}),\pi(g))$. Then, some diagram chasing yields

$$
M=\begin{cases}\Matrix{\lambda-r't}{-\lambda+r't-t}{\mu+sr'}{-\mu-sr'+s},\;r=2r'-1,\text{ $H$ not split;}\\
\Matrix{\lambda-r't}{-t}{\mu+r's}s,\;r=2r',\text{ $H$ split.}
\end{cases}
$$
It is easy to see that these matrices coincide with the ones in (2). Knowing these explicit values for $M$, the
equality in (3) is immediate.
\end{proof}

Consider now the case in which $H$ is not monic. This means that $\ell\ne2$ and $s=2s'$, $t=2t'$.

\begin{Prop}\label{extensionsnotmonic}
Assume $H$ is not monic and consider $\psi\colon \kx\to\kx\times\kx$ given by $\psi(\zeta)=(\zeta^s,\zeta^t)$. Let
$d=(q-1)/2$. There is an exact sequence of abelian groups
$$\xymatrix{\kx/\{\pm1\}\;\Mono^-{\psi'}&\kx\times\kx\Epi^-{\pi'}&\kx\times\{\pm1\}}$$
where $\psi'$ is induced by $\psi$ and $\pi'$ is given by $\pi'(\zeta,\tau)=(\zeta^{-t'}\tau^{s'},(\zeta^n\tau^m)^d)$.
\end{Prop}

\begin{proof}
 Since $H$ is not monic, we have that $q$ is odd and $(s,t)=2$. We know that $ns'+mt'=1$. The injectivity of $\psi'$ as
well as the identity $\pi'\psi'=1$ are clear. Notice that $(\zeta,1)=\pi'(\zeta^{-m},\zeta^n)$. Also, if $\theta$ is a
generator of $\kx$, we have $(1,-1)=\pi'(\theta^{s'},\theta^{t'})$. This proves the surjectivity of $\pi'$. Assume
$\pi'(\zeta,\tau)=(1,1)$. If $n$ is even, then $m$ is odd and this implies that $\tau$ is a square $\tau=\gamma^2$.
Then, $(\zeta,\tau)=\psi'(\zeta^{n/2}\gamma^m)$. The case $m$ even is similar. Finally, if $n,m$ are both odd, we have
$(\zeta,\tau)^d=1$ and so $\zeta,\tau$ are both squares or both non squares. If $\zeta=\delta^2$, $\tau=\gamma^2$, then
$(\zeta,\tau)=\psi'(\delta^n\gamma^m)$. If $\zeta$, $\tau$ were both non squares, the equality $\zeta^{t'}=\tau^{s'}$
is in contradiction to $ns'+mt'=1$ with $n,m$ odd.
\end{proof}

Notice that this exact sequence is always split over $\kx\times\{1\}$, a section being given by
$\sigma(\zeta,1)=(\zeta^{-m},\zeta^n)$. But one sees easily that this exact sequence is not always split over
$\{1\}\times\{\pm1\}$.

\begin{Prop}\label{parabolicnotmonic}
Assume $H$ is not monic. Let $\eta\colon T\to PGL_2(k)\times\kx$ given by
$\overline{\eta}(\zeta,\tau)=\left(\left[\psMatrix\zeta001\right],\tau\right)$ and let
$\overline\omega\colon \Lambda\to\Lambda$ be given by $\overline\omega(e,v)=(-e,v)$. Then
\begin{enumerate}
\item There is an isomorphism $\gamma\colon H\iso PGL_2(k)\times\kx$ such that the following diagram is commutative.
$$\xymatrix{
T\ar[r]^{\eta}\ar[d]_{\gamma|_T}&H\ar[d]_\iso^\gamma\\
T\ar[r]^(0.3){\overline\eta}&PGL_2(k)\times\kx
}
$$
 \item $\gamma|_T$ is induced by $M\colon \Lambda\to\Lambda$ given by $$M=\Matrix{n}{-t/2}{m}{s/2}.$$
\item $M^{-1}\omega M=\overline\omega$.
\end{enumerate}
\end{Prop}

\begin{proof}
Since $H$ is not monic, $\phi\colon SL_2(k)\to H$ factors through a monomorphism
$\phi'\colon PSL_2(k)\to H$ which fits into a commutative diagram
$$
\xymatrix{
PSL_2(k)\;\Mono^-{\phi'}&H\Epi^-{\pi'}&\kx\times\{\pm1\}\\
\kx/\{\pm1\}\ar[u]\;\Mono^-{\phi'}&T\ar[u]^\eta\Epi^-{\pi'}&\kx\times\{\pm1\}\ar@{=}[u]\\
\kx/\{\pm1\}\;\ar@{=}[u]\Mono^-{\psi'}&\kx\times\kx\ar[u]^\iso_\alpha\Epi^-{\pi'}&\kx\times\{\pm1\}\ar@{=}[u]
}
$$
where the bottom row coincides with the exact sequence in \ref{extensionsnotmonic}. This bottom exact sequence is split
over $\kx\times\{1\}$ and a section is $\sigma(\zeta,1)=(\zeta^{-m},\zeta^n)$. In general, it is not split over
$\{1\}\times\{\pm1\}$. However, the top exact sequence is split. To see this, consider the element
$$\beta=\phi\psMatrix0{-1}10\eta\alpha(\theta^{s'},\theta^{t'})\in H$$ where $\theta\in\kx$ is a generator. Then
$\pi'(\beta)=(1,-1)$ and an easy computation using formulas $(\ast)$ and $(\ast\ast)$ in section \ref{Tits} shows that
$\beta^2=1$. Hence, $H$ is a semidirect product $H\iso PSL_2(k)\rtimes(\kx\times\{\pm1\})$.

Then, if we compute the induced action of $\kx\times\{\pm1\}$ on $PSL_2(k)$, it turns out that $\kx$ acts trivially,
while $\{\pm1\}$ acts through $$\psMatrix xyzt\mapsto\psMatrix t{-\zeta^{-1}z}{-\zeta y}x.$$It is easy to identify this
extension. Consider the homomorphism $\epsilon\colon PGL_2(k)\to\{\pm1\}$ which sends to $-1$ all matrices whose determinant
is not a square. The kernel of $\epsilon$ is $PSL_2(k)$ and we have an extension $$\xymatrix{
PSL_2(k)\;\Mono^-{\phi'}&PGL_2(k)\Epi^-{\epsilon}&\{\pm1\}}.$$This extension has a section given by
$-1\mapsto\psMatrix0{-\zeta^{-1}}10$ and the action of $\{\pm1\}$ on $PSL_2(k)$ is exactly the same as above. This
proves that $H\iso PGL_2\times\kx$.

We want to compute now the homomorphism
$$\xymatrix{\kx\times\kx\ar[r]^-\alpha
&T\ar[r]^\eta
&H\ar[r]^-\delta
&PSL_2(k)\rtimes(\kx\times\{\pm1\})\ar[r]
&PGL_2(k)\times\kx.}$$

We omit the details of this computation which is straightforward once we have an explicit section of the extension and we recall
that $\delta$ is given by $$\delta(g)=(\phi'^{-1}(g\sigma(\pi'(g)^{-1})),\pi'(g)).$$ Once we have the matrix $M$, the equality in
(3) follows immediately.
\end{proof}

\section{The Weyl group and its invariants }\label{Weylsec}

The Weyl group $W$ of the Kac-Moody group $G$ is an infinite dihedral group.
There are two relevant integral representation of $W$. First, there is the action of $W$ on $\Q^2$ given by the root
system, as studied in section \ref{roots}. On the other side, $W$ acts on the lattice $\Lambda$ by
$\omega_i(\lambda)=\lambda-\langle h_i,\lambda\rangle\,\alpha_i$ for $i=1,2$. In this section we are interested in the
representation given by the action of $W$ on $\Lambda\otimes\padic$.

The integral $p$-adic representations of an infinite dihedral group were classified in \cite{dihedral} (see also
\cite{arboreal}) using some explicit invariants which can be easily computed. Let us denote by $\nu_p$ the $p$-adic
valuation.

\begin{Thm}\label{dihedral}
(\cite{dihedral},\cite{arboreal}) Assume $p$ is an odd prime. There are invariants $\Gamma(\rho)\in\padic$,
$\delta_1(\rho),\delta_2(\rho)\in\{0,1,2,\ldots,\infty\}$ which classify all representations $\rho\colon W\to GL_2(\padic)$. These
invariants can take any value in their range, subject to the relation $\delta_1+\delta_2=\nu_p(\Gamma)+\nu_p(\Gamma-1)$. If a
representation $\rho$ is given by $\omega_1=\psMatrix100{-1}$ and $\omega_2=M^{-1}\psMatrix100{-1}M$ with $M=\psMatrix xyzt$ (any
representation is like that), then $\Gamma(\rho)=xt/\det(M)$, $\delta_1(\rho)=\nu_p(xz)$ and $\delta_2(\rho)=\nu_p(yt)$. \qed
\end{Thm}

There is also a similar result for $p=2$ which is more complex and will not be needed in this work (cf.\
\cite{dihedral},\cite{arboreal}).

Let us apply this result to classify the action of $W$ on $\Lambda\otimes\padic$. Recall from section \ref{Tits} the
meaning of the integers $a$, $b$, $\Delta$, $\nabla$.

\begin{Prop}\label{Weyl}
The $p$-adic representation of the Weyl group of $G$ given by the action of $W$ on $\Lambda\otimes\padic$
is classified by $\Gamma=ab/4$, $\delta_1=\nu_p(a)+\nu_p(\Delta)$, $\delta_2=\nu_p(b)+\nu_p(\nabla)$.
\end{Prop}

\begin{proof}
The representation is given by the matrices
$$\omega_i=\Matrix{1-s_in_i}{-t_in_i}{-s_in_i}{1-t_in_i},\quad i=1,2.$$
The matrices $P_i=\psMatrix{t_i}{n_i}{-s_i}{m_i}$ for $i=1,2$ have determinant 2 and satisfy
$P_i^{-1}\omega_iP_i=\psMatrix100{-1}$. Hence, the matrix $M$ in \ref{dihedral} is
$M=P_2^{-1}P_1=\frac12\psMatrix{-a}\nabla{-\Delta}{-b}$.
\end{proof}

The action of $W$ on $\Lambda$ gives an action of $W$ on $T=\Hom(\Lambda,\kx)$ and on $H^1(T;\Fp)$. Of course, $H^1(T;\Fp)$
vanishes unless $q\equiv1\,(p)$ and in this case $H^1(T;\Fp)=\Fp^2$ and we have a representation of $W$ in
$GL_2(\Fp)$ which is the reduction modulo $p$ of the $p$-adic representation of $W$ studied in \ref{Weyl} and factors
through some finite dihedral group $W_p$. We are interested in these representations and in their rings of invariants.
The mod $p$ reductions of the representations of an infinite dihedral group in $GL_2(\padic)$ were studied in
\cite{Kyoto} (see Table 1 in \cite{Kyoto}). Let us summarize the results of \cite{Kyoto} that we need here. We denote by
$P$ a polynomial algebra $P=\Fp[x,y]$ with two generators of degree 2, with the given action action of $W_p$. There are
six types of representations:

\begin{itemize}
\item Type \textbf{I}. $\Gamma\con0\;(p)$, $\delta_1,\delta_2>0$. $W_p$ is elementary abelian of
order 4 and its invariants are $P^W=\Fp[x^2,y^2]$.
\item Type \textbf{II}. A representation in this type is given, up to the outer automorphism of $W$,
by $\Gamma\con0\;(p)$, $\delta_1=0,\delta_2>0$. $W_p$ is a dihedral group of order $4p$ and its invariants $P^{W_p}$ form a
polynomial ring on generators $y^2$, $x^2(x^{p-1}-y^{p-1})^2$.
\item Type \textbf{III}. $W_p$ is of order two. $\Gamma\not\con0\;(p)$,
$\delta_1,\delta_2>0$, $P^W=\Fp[x,y^2]$.
\item Type \textbf{IV}. This type occurs when $\Gamma\not\con0\;(p)$,
$\delta_1=0$, $\delta_2>0$. $W_p$ turns out to be a dihedral group of order $2p$ in $GL_2(\Fp)$
which is studied in pages 128--129 of
\cite{LarryBook}. The invariants are polynomial in degrees 2 and $4p$ namely $P^W=\Fp[x,(yx^{p-1}-y^p)^2]$.
\item Type \textbf{V}. This type occurs when $\Gamma\not\con0\;(p)$, $\delta_1>0$,
$\delta_2=0$. $W_p$ turns out to be a dihedral group of order $2p$ in $GL_2(\Fp)$
which is also studied in pages 128--129 of \cite{LarryBook}.
The invariants are polynomial in degrees 4 and $2p$ namely $P^W=\Fp[y^2,x(x^{p-1}-y^{p-1})]$.
\item Type \textbf{VI}. This type occurs when $\Gamma\not\con0\;(p)$,
$\delta_1=\delta_2=0$. The order $2m$ of the dihedral group $W_p$ is twice the multiplicative order of the roots (in
$\mathbb{F}_{p^2}$) of the polynomial $X^2-2(2\Gamma-1)X+1$. Thus, $p\con\pm1\,(m)$.
In particular, the order of $W_p$ is prime to $p$ and then
the invariants $P^W$ form a polynomial ring on generators $z_4$, $t_{2m}$ of degrees 4 and $2m$, respectively. (cf.\
\cite{Kyoto}.)
\end{itemize}

Let us study in more detail the invariant theory of the representations of type VI. What follows is probably
known, but we have not been able to find references for everything that we need in this paper. Hence, this may have
some independent interest.

\begin{Prop}\label{typeVI}
Let $\rho$ be the representation of type VI of the dihedral group of order $2m$ in $GL_2(\padic)$ classified by
$\Gamma\in\padic$. Let $\theta\in\F_{p^2}$ be a primitive $m$-th root of unity. Then, there is a basis $x,y$ such that
the ring of invariants $\Fp[x,y]^{D_{2m}}= \Fp[f,g]$ where the polynomials $f,g$ can be chosen as follows:
\begin{enumerate}
\item If $p\con1\,(m)$, then $f=xy$, $g=x^m+y^m$.
\item If $p\con-1\,(m)$, then $f=x^2+(2-4\Gamma)xy+y^2$.
\item If $p\con-1\,(m)$, let $u=\theta x-y$, $v=\theta y-x$. Then, $g=u^m+v^m$ if $m$ is even, $m\ne2$, and
$g=(u^m+v^m)/(\theta-\theta^{-1})$ if $m$ is odd.
\end{enumerate}
\end{Prop}

\begin{proof}
If $p\con1\,(m)$, then $\theta\in\Fp$ and the matrices
$$\omega=\Matrix0110,\quad\tau=\Matrix\theta00{\theta^{-1}}$$ generate the representation $\rho$. The computation of the
invariants of this representation is trivial: $$\Fp[x,y]^{D_{2m}}=\Fp[f,g],\quad f=xy,\,g=x^m+y^m.$$

If $m=p+1$, then $\theta\notin\Fp$, but $\theta^i+\theta^{-i}$ and $(\theta^i-\theta^{-i})/(\theta-\theta^{-1})$ are in
$\Fp$ for all values of $i$ (they are invariant under the Frobenius). Let $\gamma=\theta+\theta^{-1}$.
Actually, $\theta,\theta^{-1}$ are the roots of the irreducible polynomial $X^2-\gamma X+1$ and so $\theta\in\F_{p^2}$.
Also, $\gamma=4\Gamma-2$. Let us consider the matrices
$$\omega=\Matrix0110,\quad\tau=\Matrix\gamma1{-1}0,\quad\omega,\tau\in GL_2(\Fp).$$ If we extend scalars to
$\F_{p^2}$, then the change of basis given by $P=\psMatrix\theta{-1}{-1}\theta$ transforms $\omega,\tau$ into
$$P^{-1}\omega P=\Matrix0110,\quad P^{-1}\tau P=\Matrix\theta00{\theta^{-1}}$$ and we see that $\omega,\tau$ generate a
dihedral group $D_{2m}\subset GL_2(\Fp)$ which gives the representation $\rho$. Like before, the computation of the
invariants over $\F_{p^2}$ is trivial:
$$\F_{p^2}[u,v]^{D_{2m}}=\F_{p^2}[f,g],\quad f=uv,\,g=u^{p+1}+v^{p+1},$$ but we are interested in having explicit
descriptions of the invariants over $\Fp$.

We have $f=uv=-\theta(x^2-\gamma xy+y^2)$ and thus $\theta^{-1}f$ is a degree two invariant in $\Fp[x,y]^{D_{2m}}$. (Notice that
$\gamma^2-4=(\theta-\theta^{-1})^2$ is not a square in $\Fp$ and so no change of basis can produce a degree two
invariant of the form $xy$, in contrast to what is written in page 137 of \cite{LarryBook}.)

Let us consider now the invariant $g(u,v)$. If we write it in the basis $x,y$ we have
$$g=u^{p+1}+v^{p+1}=\sum_{i=0}^{p+1}(-1)^i\binom{p+1}i(\theta^i+\theta^{-i})\,x^iy^{p-i+1}.$$  Hence, $g\in\Fp[x,y]$ and
we have seen that $\Fp[x,y]^{D_{2m}}=\Fp[\theta^{-1}uv,u^{p+1}+v^{p+1}]$.

Consider now the general case $p\con-1\,(m)$, $m\ne2$. If $mr=p+1$, then $\omega,\tau^r\in GL_2(\Fp)$ produce the
representation $\rho$ of $D_{2m}$ in $GL_2(\Fp)$ and its invariants over $\F_{p^2}$ are generated by $f=uv$
and $g_m=u^m+v^m$. Like before, $x^2-\gamma xy+y^2$ is an invariant in degree two. The situation with respect to $g_m$
is slightly more complicated. If $m$ is even, then
$$g_m=u^m+v^m=\sum_{i=0}^m(-1)^i\binom{m}i(\theta^i+\theta^{-i})\,x^iy^{m-i}\in\Fp[x,y]$$ and
$\Fp[x,y]^{D_{2m}}=\Fp[x^2-\gamma xy+y^2,u^m+v^m]$ like before. On the other side, if $m$ is odd we
have $$g_m=u^m+v^m=\sum_{i=0}^m(-1)^{m-i}\binom{m}i(\theta^i-\theta^{-i})\,x^iy^{m-i}\notin\Fp[x,y]$$  and so
$g_m/(\theta-\theta^{-1})\in\Fp[x,y]$. Thus, we have $$\Fp[x,y]^{D_{2m}}=\Fp\left[x^2-\gamma
xy+y^2,\frac{u^m+v^m}{\theta-\theta^{-1}}\right].$$
\end{proof}

\section{Cohomology of the parabolic subgroups }\label{cohomologyH}

We have seen in \ref{acyclic} that the proper parabolic subgroups $P_I$ have the same mod
$p$ cohomology as their L\'{e}vi factors which are $T$, $H_1$, $H_2$. We have also seen in \ref{parabolicmonic} and
\ref{parabolicnotmonic} that the groups $H_1$, $H_2$ are isomorphic to $GL_2(k)$, $SL_2(k)\times\kx$ or
$PGL_2(k)\times\kx$. In this section we recollect some know facts about the cohomology of these finite groups.
Most of this can be found in \cite{Adem-Milgram} and \cite{Fiedor}. Recall that we are always assuming that
$p$ is an odd prime different from $\ell$.

If $q\not\equiv\pm1\;(p)$, then $p$ does not divide the order of these groups and cohomology is trivial in all cases.

Let $D=\kx\times\kx\iso(\Z/(q-1)\Z)^2$ be the subgroup of diagonal matrices in $GL_2(k)$. If
$q\con1\;(p)$, then $H^*(BGL_2(k);\Fp)$ is a ring of invariants, namely
$$H^*(BGL_2(k);\Fp)\iso H^*(BD;\Fp)^{C_2}$$ where $C_2$ is a cyclic group of order two acting on $D$ by permuting the
two factors. Also, if $r$ is such that $q-1=sp^r$ with $s$ prime to $p$, then $H^*(BD;\Fp)\iso H^*(B\Z/p^r\Z\times
B\Z/p^r\Z;\Fp)$ and $$H^*(BGL_2(k);\Fp)\iso\Fp[x_2,x_4]\otimes E(y_1,y_3)$$(the subscripts denote the degrees of the
generators) with secondary Bocksteins of height $r$ relating the exterior generators to the polynomial generators.

If $q\con-1\;(p)$, then $H^*(BD;\Fp)$ is trivial but $H^*(BGL_2(k);\Fp)$ is also a ring of invariants. If $q+1=sp^r$
with $s$ prime to $p$, then $$H^*(BGL_2(k);\Fp)\iso H^*(B\Z/p^r\Z;\Fp)^{C_2}\iso\Fp[x_4]\otimes E(y_3)$$ where $C_2$
acts on $\Z/p^r\Z$ by multiplication by $-1$.

The case of $SL_2(k)=Sp_2(k)$ is similar. If $q\equiv-1\;(p)$ then
$$H^*(B(SL_2(k)\times\kx);\Fp)\iso H^*(BSL_2(k);\Fp)\iso H^*(BGL_2(k);\Fp).$$
If $q\equiv1\;(p)$ then $$H^*(BSL_2(k);\Fp)\iso H^*(B(D\cap SL_2(k));\Fp)^{C_2}\iso\Fp[x_4]\otimes E(y_3)$$
where $C_2$ acts on $D\cap SL_2(k)\iso\Z/(q-1)\Z$ as multiplication by $-1$.

Let us consider now the case of $PGL_2(k)$. If $q\con-1\;(p)$ then $H^*(B\kx;\Fp)$ is trivial and the extension
$\kx\rightarrowtail GL_2(k)\twoheadrightarrow PGL_2(k)$ gives an isomorphism $$H^*(BPGL_2(k);\Fp)\iso
H^*(BGL_2(k);\Fp).$$ On the other side, if $q\con1\;(p)$ and $r$ is such that $q-1=sp^r$ with $s$ prime to $p$, then the
Sylow $p$-subgroup of $PGL_2(k)$ is cyclic $$Syl=\left\{\left[\psMatrix
x001\right]\;|\;x\in\kx,\,x^{p^r}=1\right\}\iso\Z/p^r\Z$$ and it is easy to compute the
stable elements in $H^*(BSyl;\Fp)$. It turns out that the only automorphism of $Syl$ that appears is the involution
$x\mapsto x^{-1}$ and so $$H^*(BPGL_2(k);\Fp)\iso H^*(BSyl;\Fp)^{C_2}\iso\Fp[x_4]\otimes E(y_3).$$

If we apply all this to the parabolic subgroups $H_1$, $H_2$ in $G$, we obtain the following.

\begin{Prop}\label{H*H}
Let $H$ be one of the L\'{e}vi subgroups of $G$ and let $\omega\in W$ be the corresponding generating reflection. Then,
$H^*(BH;\Fp)$ is trivial unless $q\con\pm1\;(p)$. If $q\con1\;(p)$ then $H^*(BH;\Fp)\iso
H^*(BT;\Fp)^\omega$. If $q\con-1\;(p)$ then $H^*(BT;\Fp)$ is trivial and $H^*(BH;\Fp)\iso\Fp[x_4]\otimes E(y_3)$ with a
Bockstein relating $y_3$ to $x_4$.
\end{Prop}

\begin{proof}
This follows from \ref{parabolicmonic}, \ref{parabolicnotmonic} and the discussion above.
\end{proof}

In this paper we are assuming that $p$ is an odd prime. It happens that the situation at the prime two is far more complex, even
for the L\'{e}vi factors. For example, if $p=2$ and $q\con3\;(4)$, we have $H^*(BD;\Fp)\iso H^*(B\Z/2\Z\times
B\Z/2\Z;\F_2)\iso\mathbb{F}_2[u_1,v_1]$ and $H^*(BGL_2(k);\mathbb{F}_2)$ is the subalgebra of $H^*(BD;\mathbb{F}_2)$ generated by
the elements
\begin{align*}
x_1&=u_1+v_1\\
x_3&=u_1v_1^2+u_1^2v_1\\
x_4&=u_1^2v_1^2
 \end{align*}
(see \cite{Fiedor}, page 342). Then, $H^*(BGL_2;\mathbb{F}_2)$ is neither a polynomial ring nor a ring of invariants.
Actually, $$H^*(BGL_2(k);\mathbb{F}_2)\iso\mathbb{F}_2[x_4,x_1,x_3]/(x_3^2+x_4x_1^2).$$

\section{The structure of $\Coker(\phi)$ }\label{Coker}

In this section we will continue the study of the invariant theory of a dihedral group of type VI that we introduced in section
\ref{Weylsec}. So, $p$ is an odd prime and $D$ is a dihedral group with a representation in $GL_2(\Fp)$ of type VI, classified by
$\Gamma\in\Fp$, $\Gamma\ne0$. The order of $D$ is $2m$ where $m$ is the order of the roots of $X^2-2(2\Gamma-1)X+1$. Depending on
the values of $a$, $b$, $\Delta$, $\nabla$ (see Proposition \ref{Weyl}), the Weyl group of the Kac-Moody group $G$
acting on $H^1(T;\Fp)$ provides an example of such a representation, but this section is written independently of the
theory of Kac-Moody groups.

We use the following notation. $P=\Fp[v,v']$ is a polynomial algebra on two generators of degree two. $S=\Fp[v,v']\otimes
E(u,u')$ is the tensor product of $P$ and an exterior algebra on two generators of degree one. We consider $P$ as a subalgebra of
$S$. From a topological point of view, we can think of $P$ and $S$ as $H^*(BS^1\times BS^1;\Fp)$ and $H^*(B\Z/p\times
B\Z/p;\Fp)$, respectively.

Let us consider the following (graded) derivations acting on $S$. $d$ has degree $-1$ and is given by
$$d(u)=d(u')=0,\;d(v)=u,\;d(v')=u'.$$ $\delta$ has degree $+1$ and is given by
$$\delta(v)=\delta(v')=0,\;\delta(u)=v,\;\delta(u')=v'.$$  Under the identification
$S\iso H^*(B\Z/p\times B\Z/p;\Fp)$, $\delta$ corresponds to the Bockstein homomorphism. It is easy to see that if $x\in P$ is a
polynomial of degree $2n$, then $\delta d(x)=nx$.

$D$ acts on the algebras $P$ and $S$ and it is clear that the derivations $d$ and $\delta$ commute with this action. From
\ref{typeVI} we know that the ring of invariants of $P$ under $D$ is a polynomial algebra $P^D=\Fp[x_4,x_{2m}]$ and we have
explicit descriptions of the generators $x_4, x_{2m}$. If we consider the action of $D$ on $S$, theorem 9.3.2 in \cite{LarryBook}
says that the ring of invariants is
$$S^D=\Fp[x_4,x_{2m}]\otimes E(dx_4,dx_{2m}).$$ $D$ is generated by two elements $\omega_1,\,\omega_2$ of order two and we
can also consider the rings of invariants of each generating reflection $S^{\omega_1}$, $S^{\omega_2}$.

Let $\phi$ be the $P^D$ linear map $$S^{\omega_1}\oplus S^{\omega_2}\to S$$ defined by $\phi(t,h)=t+h$. We want to investigate
the structure of $\mathcal{S}=\Coker(\phi)$ as a $P^D$-module.

Let us introduce now two \emph{relative invariants} $\alpha$, $J$ of $D$ in $S$. $$\alpha=uu',\quad
J=\left|\begin{matrix}\frac{\partial x_4}{\partial v}&\frac{\partial x_4}{\partial v'}
\\\frac{\partial x_{2m}}{\partial v}&\frac{\partial x_{2m}}{\partial v'}\end{matrix}\right|.$$Of course, $J$ is the
Jacobian of the algebraically independent polynomials $x_4$ and $x_{2m}$ and so it does not vanish. It is a polynomial
of degree $2m$ in $P$. One sees immediately that both $\alpha$ and $J$ are relative invariants of $D$ with respect to
the determinant: for any $g\in D$ we have $g\alpha=\det(g)\,\alpha$ and $gJ=\det(g)\,J$.

Since the derivations $d$ and $\delta$ commute with the action of $D$ on $S$, we can define two new relative invariants as
follows. $$\alpha'=\delta(\alpha)=vu'-v'u,\quad J'=\frac1m\,d(J).$$ Notice that $\delta(J')=J$.

\begin{Prop}\label{independent}
 $[\alpha]$, $[\alpha']$, $[J]$, $[J']$ are $P^D$-linearly independent elements in $\mathcal{S}$.
\end{Prop}

\begin{proof}
 Let $f[\alpha']+g[J]=0$ in $\mathcal{S}$ with $f,g\in P^D$. This means that there are $p_1\in S^{\omega_1}$, $p_2\in
S^{\omega_2}$ such that $$f\alpha'+gJ=p_1+p_2.$$applying $\omega_1$ and subtracting we get
$$2f\alpha'+2gJ=p_2-\omega_1p_2.$$ Applying $\omega_2$ we get $$-2f\alpha'-2gJ=p_2-\omega_2\omega_1p_2.$$ Notice that
$\alpha'$ and $J$ are left invariant by $\omega_2\omega_1$. Hence, applying $\omega_2\omega_1$ and adding, we get
$$-4f\alpha'-4gJ=p_2-(\omega_2\omega_1)^2p_2.$$ Inductively, we get $$-2mf\alpha'-2mgJ=p_2-(\omega_2\omega_1)^mp_2=0.$$
Since $m$ is prime to $p$, this yields $f\alpha'+gJ=0$ in $S$. But $gJ\in P$ and so $f\alpha'\in P$ and this can only
happen if $f,g=0$.

We have proven that $[\alpha']$ and $[J]$ are $P^D$-linearly independent in $\mathcal{S}$. Consider now a vanishing linear
combination $f[\alpha]+g[\alpha']+h[J]+l[J']=0$. This means that there are $p_1\in S^{\omega_1}$, $p_2\in S^{\omega_2}$ such that
$$f\alpha+g\alpha'+hJ+lJ'=p_1+p_2.$$ If we apply the derivation $\delta$ we get
$$f\alpha'+hJ=\delta(p_1)+\delta(p_2)=0$$ and the argument above yields $f=h=0$. Then, $g\alpha'+lJ=p_1+p_2$ and again
this implies $g=l=0$.
\end{proof}

The following consequence of the proposition above is a key step in our computation of the cohomology of the Kac-Moody
group $G$.

\begin{Thm}\label{free}
$\mathcal{S}$ is a free $P^D$-module with basis $\{[\alpha],[\alpha'],[J],[J']\}$.
\end{Thm}

\begin{proof}
 We know that the free $P^D$-module with basis $\{[\alpha],[\alpha'],[J],[J']\}$ is a submodule of $\mathcal{S}$. It is
enough to prove that both graded vector spaces have the same Poincar\'{e} series. The kernel of $\phi$ is $S^D$. It is very easy
to write down the Poincar\'{e} series of $S$, $S^{\omega_1}\oplus S^{\omega_2}$ and $S^D$. Then, a straightforward computation
with power series proves that the Poincar\'{e} series of $\mathcal{S}$ is as expected.
\end{proof}

The computation of the multiplicative structure of $H^*(BG;\Fp)$ in the next section requires that we work out some
relations that hold in $S$. Let $y_3=dx_4$, $y_{2m-1}=dx_{2m}$. Let $\theta,\gamma$ be as in the proof of \ref{typeVI}.

\begin{Thm}\label{Jacobians}
 The following identities hold in $S$:
\begin{enumerate}
 \item $y_3J'=m\lambda x_{2m}\alpha$, $y_{2m-1}J'=2m^2\mu x_4^{m-1}\alpha$.
\item $y_3J=2x_4J'-m\lambda x_{2m}\alpha'$, $y_{2m-1}J=mx_{2m}J'-2m^2\mu x_4^{m-1}\alpha'$.
\end{enumerate}
where the coefficients $\lambda$ and $\mu$ are as follows:
\begin{enumerate}
 \item[(a)] $\lambda=1$ if $p\equiv1\,(m)$; $\lambda=\gamma^2-4$ if $p\equiv-1\,(m)$.
\item[(b)] $\mu=1$ if $p\equiv1\,(m)$; $\mu=\gamma^2-4$ if $p\equiv-1\,(m)$ and $m\ne2$ is even; $\mu=-1$ if
$p\equiv-1\,(m)$ and $m$ is odd.
\end{enumerate}
\end{Thm}

\begin{proof}
 The identities in (2) follow from (1) applying the derivation $\delta$. To prove (1) consider
$$y_3J'=\frac1m\,dx_4\,dJ=\frac1m\mathcal{J}(x_4,\mathcal{J}(x_4,x_{2m}))\alpha$$
$$y_{2m-1}J'=\frac1m\,dx_{2m}\,dJ=\frac1m\mathcal{J}(x_{2m},\mathcal{J}(x_4,x_{2m}))\alpha$$
where $\mathcal{J}(f,g)$ denotes the determinant of the Jacobian matrix of the polynomials $f,g$. To compute these
double Jacobians, notice that the Jacobian is a relative invariant of the determinant and so if we change bases then
the Jacobian is multiplied by the determinant of the change of basis matrix.

We will use the descriptions of $x_4$ and $x_{2m}$ as given in \ref{typeVI}.

If $p\equiv1\,(m)$, then $x_4=vv'$, $x_{2m}=v^m+(v')^m$. Then, a trivial direct calculation gives
$$\mathcal{J}(x_4,\mathcal{J}(x_4,x_{2m}))=m^2x_{2m}$$$$
\mathcal{J}(x_{2m},\mathcal{J}(x_4,x_{2m}))=2m^3x_4^{m-1}.$$

If $p\equiv-1\,(m)$, we extend scalars as to have the $m$-th root of unity $\theta$ in our field of coefficients. Then $x_4$ and
$x_{2m}$ have simple descriptions in a new base $w,w'$ which is related to the basis $v,v'$ through a matrix $P$ of determinant
$\theta^2-1$ (see the proof of \ref{typeVI}).

If $p\equiv-1\,(m)$ and $m$ is even, $m\ne2$, then $x_4=-\theta^{-1}ww'$ and $x_{2m}=w^m+(w')^m$. In this new basis, the
computation of the double Jacobians is as easy as in the case $p\equiv1\,(m)$. We get
$$\mathcal{J}(x_4,\mathcal{J}(x_4,x_{2m}))=\theta^{-2}(\theta^2-1)^2m^2x_{2m}=m^2(\gamma^2-4)x_{2m}$$$$
\mathcal{J}(x_{2m},\mathcal{J}(x_4,x_{2m}))=-\theta^{-1}(\theta^2-1)^22m^3(ww')^{m-1}=2m^3(\gamma^2-4)x_4^{m-1}.$$

If $p\equiv-1\,(m)$ and $m$ is odd, then $x_4=-\theta^{-1}ww'$ and $x_{2m}=(w^m+(w')^m)/(\theta-\theta^{-1})$. Then
$$\mathcal{J}(x_4,\mathcal{J}(x_4,x_{2m}))=\theta^{-2}(\theta-\theta^{-1})^{-1}
(\theta^2-1)^2m^2(w^m+(w')^m)=m^2(\gamma^2-4)x_{2m}$$$$
\mathcal{J}(x_{2m},\mathcal{J}(x_4,x_{2m}))=-\theta^{-1}(\theta-\theta^{-1})^{-2}(\theta^2-1)^22m^3(ww')^{m-1} =-2m^3x_4^ { m-1 }
.$$
\end{proof}

\section{The cohomology of $BG$}\label{cohomologyG}

The preceding sections have provided all ingredients that we need to compute the cohomology algebra of the Kac-Moody group $G$
with coefficients in the field $\Fp$. By \ref{Mitchell} we know that $BG$ is homotopy equivalent to
the push out of $BP_1\leftarrow BP_{\varnothing}\rightarrow BP_2$. By \ref{acyclic} we know that $H^*(BG;\Fp)\iso
H^*(B(H_1\ast_TH_2);\Fp)$. We have computed $H_1$ and $H_2$ in section \ref{parabolicsec} and we have computed
$H^*(BH_i;\Fp)$ in \ref{H*H}. Then, the obvious tool to use to compute $H^*(B(H_1\ast_TH_2);\Fp)$ is the Mayer-Vietoris
exact sequence. The following lemma on the behavior of the connecting homomorphism of the Mayer-Vietoris exact sequence
will be useful.

\begin{Lem}\label{Dold}
 Let
$$\xymatrix{
A\ar[r]^{i_1}\ar[d]^{i_2}\ar[dr]^l&B_1\ar[d]^{j_1}\\
B_2\ar[r]^{j_2}&X
}
$$
be a push out diagram and let $\Delta\colon H^*(A)\to H^*(X)$ be the connecting homomorphism of the corresponding
Mayer-Vietoris exact sequence (coefficient ring omitted). Then, for any $x\in H^*(X)$, $t\in H^*(A)$, we have
$\Delta(l^*(x)\,t)=x\,\Delta(t)$.
\end{Lem}

\begin{proof}
Let us assume that all maps in the diagram are cofibrations. Then, $\Delta$ is defined through this commutative diagram
(cf.\ \cite{Dold}, p.\ 49):
$$\xymatrix{
H^*(A)\ar[r]^{\Delta}\ar[d]^{\delta}&H^*(X)\\
H^*(B_1,A)&\ar[l]_{k^*}^{\cong}H^*(X,B_2)\ar[u]^{r^*}
}
$$
where $r$ and $k$ are the obvious inclusions, $\delta$ is the connecting homomorphism for the couple $(B_1,A)$ and $k$
is an isomorphism because of excision. The way that $\delta$ behaves with respect to cup products (cf.\ \cite{Dold},
p.\ 220) gives $$\delta(i_1^*(b_1)\,t)=b_1\,\delta(t),\quad\text{for }b_1\in H^*(B_1),\;t\in H^*(A).$$ Then, an easy
calculation using $\Delta=r^*(k^*)^{-1}\delta$ gives the formula of the lemma.
\end{proof}

Let us deal first with the easy case in which $q\not\equiv1\;(p)$. If $(A_1,\epsilon_1)$ and $(A_2,\epsilon_2)$ are augmented
graded $\Fp$-algebras, let us denote by $A_1\vee A_2$ the augmented graded algebra defined as the kernel of
$\epsilon_1-\epsilon_2\colon A_1\oplus A_2\to\Fp$.

\begin{Prop}\label{trivial}
 If $q\not\equiv\pm1\;(p)$ then $H^*(BG;\Fp)\iso\Fp$. If $q\equiv-1\;(p)$ then $H^*(BG;\Fp)\iso A\vee A$ with
$A=\Fp[x_4]\otimes E(y_3)$ (subscripts denote degrees).
\end{Prop}

\begin{proof}
 This follows immediately from the Mayer-Vietoris exact sequence
$$\xymatrix@C=20pt{\cdots\ar[r]&H^{*-1}(BT)\ar[r]^{\Delta}&H^*(BG)\ar[r]^(.35){\psi}&H^*(BH_1)\oplus
H^*(BH_2)\ar[r]^(.66){\phi}&H^*(BT) \ar[r]&\cdots}$$ and Proposition \ref{H*H}.
\end{proof}

Of course, the relevant case is when $q\equiv1\;(p)$. In this case, the computation of $H^*(BG;\Fp)$ will be made under
the following \textbf{stability hypothesis}:

\medskip
\textbf{(SH)} The prime $p$ is large enough so that $ab(ab-4)$ is prime to $p$.

\medskip
This hypothesis implies that the representation of the Weyl group $W$ of $G$ is of Type VI (cf.\ \ref{Weyl}). We use
the notations of section \ref{Coker}. In particular $S^W=\Fp[x_4,x_{2m}]\otimes E(y_3,y_{2m-1})$ and
$P^W=\Fp[x_4,x_{2m}]$. We are ready for the main result of this paper:

\begin{Thm}\label{Main}
Assume that $q\equiv1\;(p)$ and the stability hypothesis (SH) is satisfied. Then $H^*(BG;\Fp)$ is a $S^W$-module with
generators $1,\alpha_3,\alpha_4,J_{2m},J_{2m+1}$ (subscripts denote degrees) and relations
\begin{enumerate}
\item $y_3\alpha_3=0,\;y_{2m-1}\alpha_3=0$.
\item $2x_4\alpha_3-y_3\alpha_4=0,\;mx_{2m}\alpha_3-y_{2m-1}\alpha_4=0$
\item $y_3J_{2m}-m\lambda x_{2m}\alpha_3=0,\;y_{2m-1}J_{2m}-2m^2\mu x_4^{m-1}\alpha_3=0$
\item $y_3J_{2m+1}+m\lambda x_{2m}\alpha_4-2x_4J_{2m}=0,\;y_{2m-1}J_{2m+1}+2m^2\mu x_4^{m-1}\alpha_4-mx_{2m}J_{2m}=0$
\end{enumerate}
where the parameters $\lambda$ and $\mu$ are as follows:
\begin{enumerate}
\item If $p\equiv1\,(m)$, then $\lambda=\mu=1$ .
\item If $p\equiv-1\,(m)$, then $\lambda=ab(ab-4)$ and $\mu=ab(ab-4)$ if $m\ne2$ is even and $\mu=-1$ if $m$ is odd.
\end{enumerate}
The algebra structure of $H^*(BG;\Fp)$ is determined by the above relations
and the fact that all products between the generators $\alpha_3,\alpha_4,J_{2m},J_{2m+1}$ vanish.
\end{Thm}

\begin{proof}
As discussed above, we have a Mayer-Vietoris exact sequence (coefficients in $\Fp$ are assumed)
$$\xymatrix@C=20pt{\cdots\ar[r]&H^{*-1}(BT)\ar[r]^{\Delta}&H^*(BG)\ar[r]^(.35){\psi}&H^*(BH_1)\oplus
H^*(BH_2)\ar[r]^(.66){\phi}&H^*(BT) \ar[r]&\cdots}$$
By \ref{H*H} we know that $H^*(BH_i)\iso S^{\omega_i}$, $i=1,2$, where $\omega_1,\omega_2$ are the generators of the
Weyl group $W$. Then, the kernel of $\phi$ is $S^W$ and if we denote $\mathcal{S}=\Coker(\phi)$ as in section
\ref{Coker}, we have a short exact sequence
$$\xymatrix{0\ar[r]&\mathcal{S}\ar[r]^(.35){\overline{\Delta}}&H^*(BG)\ar[r]^(.6){\overline{\psi}}&S^W\ar[r]&0}.$$
Notice that $\overline{\psi}$ is a ring homomorphism and $S^W$ is a free graded algebra. Hence, we can choose a section
$\sigma$ of $\phi$ which is a ring homomorphism. This section turns $H^*(BG)$ into an $S^W$-module. Then, lemma
\ref{Dold} implies that $\Delta$ is $S^W$-linear. We have that the above short exact sequence is an exact sequence of
$P^W$-modules. Both $\mathcal{S}$ and $S^W$ are free $P^W$-modules (Proposition \ref{free}) and so $H^*(BG)$ is a free
$P^W$-module with basis
\begin{gather*}
1,\sigma(y_3),\sigma(y_{2m-1}),\sigma(y_3)\sigma(y_{2m-1}),\\ \alpha_3=\overline{\Delta}[\alpha],
\alpha_4=\overline{\Delta}[\alpha'],
J_{2m+1}=\overline{\Delta}[J],
 J_{2m}=\overline{\Delta}[J'].
\end{gather*}

To complete the proof we should compute all products between these generators. First of all, notice that $\Delta$ is a connecting
homomorphism which arises from a map $BG\to\Sigma BT$ and so all products in the image of $\Delta$ vanish. In $S$ we have
$y_3\alpha=y_{2m-1}\alpha=0$ and this implies the relations (1) in $H^*(B\G)$. Applying $\delta$ we get the relations (2).
Relations (3) and (4) follow from the identities in $S$ proven in Theorem \ref{Jacobians} and the fact that $\gamma^2-4=ab(ab-4)$
in this case. \end{proof}

\section{Comparison between $BG$ and the fixed points by an Adams map}\label{Adams}

There is a well know result by Friedlander (\cite{Friedlander}) that relates the classifying space of a Chevalley group
over a finite field $k$ to the homotopy fixed points of and unstable Adams map from the classifying space of the
corresponding compact connected Lie group, after completion at a prime different to the characteristic of $k$. More
precisely, let $p,\ell$ be different primes, $k$ a finite field of order $q$ and characteristic $\ell$, $K$ a compact
connected Lie group and $K(q)$ the Chevalley group over $k$ of type $K$. Then, there is a homotopy pullback diagram
$$\xymatrix{
BK(q)\pcom\ar[r]\ar[d]&BK\pcom\ar[d]^{\Delta}\\
BK\pcom\ar[r]^(.4){(1,\psi^q)}&BK\pcom\times BK\pcom
}$$
where $\psi^q$ is an unstable Adams map of exponent $q$. This result was generalized in \cite{BrotoMoller} to the case
in which $K$ is a $p$-compact group. Then, $K(q)$ turns out to be a $p$-local finite group. In this final section we
want to investigate if such a pullback diagram could also exist in the context of Kac-Moody groups.

Let $K$ denote the unitary form of a connected, simply connected, not compact Kac-Moody group as in \cite{KacMSRI},
\cite{tesiNitu}, \cite{CarlesNitu} and assume that $K$ is of rank two, uniquely determined by two positive integers $a,b$ with
$ab>4$. In the case of rank two, there is a satisfactory theory of Adams maps $\psi^q\colon BK\to BK$ developed in \cite{Advances} and
so, for a prime $p$ and a prime power $q$ with $p\nmid q$, it makes sense to consider the space $X$ defined as the homotopy
pullback
$$\xymatrix{
X\ar[r]\ar[d]&BK\pcom\ar[d]^{\Delta}\\
BK\pcom\ar[r]^(.4){(1,\psi^q)}&BK\pcom\times BK\pcom
}$$

\emph{If} Friedlander theorem were true for Kac-Moody groups, then this space $X$ should have the same mod
$p$ cohomology as the classifying space of the discrete Kac-Moody group $G_{\mathcal{D}}(\F_q)$ associated to the data
$\mathcal{D}$ consisting of the generalized Cartan matrix $A=\psMatrix2{-a}{-b}2$ and the matrix
$\psMatrix{n_1}{n_2}{m_1}{m_2}=\psMatrix1001$ (see section \ref{Tits}). Interestingly enough, this is \emph{not} true.

Let us choose the data $a,b,p,q$ such that the following holds. $p$ and $q$ are odd and relatively prime. $p$ does not
divide $ab(ab-4)$, $q\not\equiv\pm1\,(p)$ and $q^m\equiv1\,(p)$ where $m$ is the multiplicative order in $\F_{p^2}$ of
the roots of $X^2+(2-ab)X+1$. It is easy to see that these choices are possible. Under these hypothesis, the mod
$p$ cohomology of the classifying space of the (discrete) Kac-Moody group $G$ is trivial (Proposition \ref{trivial}).

The mod $p$ cohomology of the classifying space of the unitary, connected, simply
connected, Kac-Moody group $K=K(a,b)$ is computed in \cite{Kyoto}:
$$H^*(BK;\Fp)=\F_p[x_4,y_{2m}]\otimes E(z_{2m+1}).$$

Let $\psi^q\colon BK\to BK$ be an unstable Adams map as defined in \cite{Advances}. When restricted to a maximal torus of
$K$, $\psi^q$ is the $q$-power map. Let $X$ be defined by the homotopy pullback above. We will show that the mod $p$
cohomology of $X$ does not vanish in positive degrees.

Consider the Serre spectral sequence in mod $p$ cohomology for the fibration sequence
$$\xymatrix{K\ar[r]&BK\ar[r]^(.4){\Delta}&BK\times BK.}$$
The cohomology of the topological group $K$ is as follows (\cite{tesiNitu}):
$$H^*(K;\Fp)=E(\tilde{x}_3,\tilde{y}_{2m-1})\otimes\Gamma[\tilde{z}_{2m}].$$
($\Gamma$ denotes a divided power algebra.) Then, it is obvious that the elements $\tilde{x}_3$ and
$\tilde{y}_{2m-1}$ in the spectral sequence have to be transgressive to $x_4\otimes1-1\otimes x_4$ and
$y_{2m}\otimes1-1\otimes y_{2m}$, respectively. Then, pulling back by $(1,\psi^q)^*$ to the Serre spectral sequence for
$\xymatrix{K\ar[r]&X\ar[r]&BK}$ we see that in this spectral sequence the element $\tilde{y}_{2m-1}$ is transgressive
to $(1-q^m)y_{2m}=0$ and so it survives to $H^{2m-1}(X;\Fp)$.

A similar phenomenon happens also in the case $q\equiv\pm1\,(p)$.

This discrepancy between classifying spaces of infinite Kac-Moody groups over finite fields and the homotopy fixed
points of the corresponding Adams maps deserves further study.

\bibliographystyle{alpha}
\bibliography{biblio} 

\end{document}